\documentclass[12pt,leqno]{article}

\usepackage{amssymb,amsmath,amsthm}
\usepackage[all]{xy}
\numberwithin{equation}{section}

\usepackage{enumerate}
\usepackage{mathrsfs}
\usepackage{color}
\usepackage[colorlinks=true, pdfstartview=FitV, linkcolor=blue,%
citecolor=blue, urlcolor=blue]{hyperref}

\newcommand{\nc}{\newcommand}
\nc{\on}{\operatorname}

\newtheorem{theorem}{Theorem}[section]
\newtheorem{proposition}[theorem]{Proposition}
\newtheorem{lemma}[theorem]{Lemma}
\newtheorem{corollary}[theorem]{Corollary}

\theoremstyle{definition}

\newtheorem{definition}[theorem]{Definition}

\newtheorem{example}[theorem]{Example}

\newtheorem{remark}[theorem]{Remark}

\nc{\RR}{\mathrm{R}}
\nc{\LL}{\mathrm{L}}
\newcommand{\RC}{{\rm C}}

\newcommand{\C}{{\mathbb{C}}}
\newcommand{\N}{{\mathbb{N}}}

\newcommand{\Z}{{\mathbb{Z}}}

\def\phi{{\varphi}}
\def\epsilon{\varepsilon}

\newcommand{\cor}{{\bf k}}

\newcommand{\bfone}{{\bf 1}}

\def\shc{\mathscr{C}}

\def\shd{\mathscr{D}}

\def\she{\mathscr{E}}
\def\shf{\mathscr{F}}

\def\shm{\mathscr{M}}
\def\shn{\mathscr{N}}
\def\sho{\mathscr{O}}
\def\shp{\mathscr{P}}

\def\shr{\mathscr{R}}

\def\sht{\mathscr{T}}

\renewcommand{\ker}{\operatorname{Ker}}
\DeclareMathOperator{\coker}{Coker}
\DeclareMathOperator{\im}{Im}
\DeclareMathOperator{\coim}{Coim}
\DeclareMathOperator{\Rees}{\Sigma}
\DeclareMathOperator{\LH}{LH}

\newcommand{\rmptt}{{\{\rm pt\}}}
\newcommand{\into}{\hookrightarrow}

\def\epito{{\twoheadrightarrow}}
\renewcommand{\to}[1][]{\xrightarrow[]{#1}}

\newcommand{\isoto}[1][]{\xrightarrow[#1]%
{{\raisebox{-.6ex}[0ex][-.6ex]{$\mspace{1mu}\sim\mspace{2mu}$}}}}

\newcommand{\muHom}[1][]{\mathrm{Hom}^\mu_{\raise1.5ex\hbox to.1em{}#1}}
\newcommand{\Hom}[1][]{\mathrm{Hom}_{\raise1.5ex\hbox to.1em{}#1}}
\newcommand{\RHom}[1][]{\RR\mathrm{Hom}_{\raise1.5ex\hbox to.1em{}#1}}
\newcommand{\Ext}[2][]{\mathrm{Ext}_{\raise1.5ex\hbox to.1em{}#1}^{#2}}
\renewcommand{\hom}[1][]{{\mathscr{H}\mspace{-4mu}om}_{\raise1.5ex\hbox to.1em{}#1}}
\newcommand{\rhom}[1][]{{\RR\mathscr{H}\mspace{-3mu}om}_{\raise1.5ex\hbox to.1em{}#1}}
\newcommand{\rhomc}[1][]
{{\mathscr{H}\mspace{-3mu}om}^*_{\raise1.5ex\hbox to.1em{}#1}}

\newcommand{\ext}[2][]{{\mathscr{E}xt}_{\raise1.5ex\hbox to.1em{}#1}^{#2}}
\newcommand{\Tor}[2][]{\mathrm{Tor}^{\raise1.5ex\hbox to.1em{}#1}_{#2}}
\newcommand{\tens}[1][]{\mathbin{\otimes_{\raise1.5ex\hbox to-.1em{}{#1}}}}

\newcommand{\Endo}[1][]{\mathrm{End}_{\raise1.5ex\hbox to.1em{}#1}}

\newcommand{\Aut}[1][]{\mathrm{Aut}_{\raise1.5ex\hbox to.1em{}#1}}

\newcommand{\md[1]}{{\rm Mod}({#1})}

\newcommand{\Mod}{\mathrm{Mod}}

\newcommand{\eqdot}{\mathbin{:=}}

\newcommand{\cl}{\colon}
\newcommand{\scbul}{{\,\raise.4ex\hbox{$\scriptscriptstyle\bullet$}\,}}

\newcommand{\lp}{{\rm(}}
\newcommand{\rp}{{\rm)}}

\newcommand{\ba}{\begin{array}}
\newcommand{\ea}{\end{array}}

\newcommand{\bnum}{\begin{enumerate}[{\rm(i)}]}
\newcommand{\enum}{\end{enumerate}}
\newcommand{\banum}{\begin{enumerate}[{\rm(a)}]}
\newcommand{\eanum}{\end{enumerate}}

\newcommand{\eq}{\begin{eqnarray}}
\newcommand{\eneq}{\end{eqnarray}}
\newcommand{\eqn}{\begin{eqnarray*}}
\newcommand{\eneqn}{\end{eqnarray*}}

\nc{\Proof}{\begin{proof}}
\nc{\QED}{\end{proof}}

\def\rop{{\rm op}}

\def\Xsa{{X_{\rm sa}}}

\DeclareMathOperator{\id}{id}

\DeclareMathOperator{\For}{for}

\newcommand{\Der}[1][]{\mathsf{D}^{#1}}
\newcommand{\Derb}{\Der[\mathrm{b}]}

\newcommand{\RD}{\mathrm{D}}
\newcommand{\rb}{{\mathrm b}}
\newcommand{\ub}{\mathrm{ub}}
\newcommand{\Mc}{\mathrm{Mc}}
\newcommand{\rK}{{\mathrm K}}
\newcommand{\qis}{\mathrm {qis}}

\newcommand{\mon}{\Lambda}

\newcommand{\Fct}{\operatorname{Fct}}
\newcommand{\Fil}{\operatorname{F}}

\newcommand{\dT}{{\dot{T}}}

\newcommand{\indlim}[1][]{\mathop{\varinjlim}\limits_{#1}}
\newcommand{\sindlim}[1][]{\smash{\mathop{\varinjlim}\limits_{#1}}\,}

\nc{\eps}{\varepsilon}
\nc{\hs}{\hspace*}
\nc{\nn}{\nonumber}
\nc{\tM}{\widetilde{M}}
\nc{\h}{\mathbf{h}}
\nc{\tf}{\tilde{f}}
\nc{\codim}{\on{codim}}
\nc{\lh}{\mathscr{H}}
\nc{\bwr}{\mbox{\large{$\wr$}}}
\nc{\dTi}{\dT^{*,\mathrm{in}}}
\nc{\Cd}{\mathrm{C}}
\nc{\Filt}{\mathrm{Filt}}

\newcommand{\comp}{\mathbin{\circ}}

\newenvironment{rouge}
{\relax\color{red}}
{\hspace*{.3ex}\relax}
\newcommand{\ber}{\begin{rouge}}
\newcommand{\er}{\end{rouge}}

\newenvironment{bleu}
{\relax\color{blue}}
{\hspace*{.3ex}\relax}
\newcommand{\beb}{\begin{bleu}}
\newcommand{\eb}{\end{bleu}}
\begin{document}

\title{Derived category of filtered objects}
\author{Pierre Schapira and Jean-Pierre Schneiders} 
\maketitle

\begin{abstract}
For an abelian category $\shc$ and a filtrant preordered set $\mon$, we prove that the derived category of the quasi-abelian category of filtered objects in $\shc$ indexed by $\mon$ is equivalent to the derived category of the abelian category of functors from 
$\mon$ to $\shc$. We apply this result to the study of the category of filtered modules over a filtered ring in a tensor category. 
\end{abstract}
\section{Introduction}
Filtered modules over filtered sheaves of rings appear naturally in mathematics, such as for example when studying 
$\shd_X$-modules on a complex manifold $X$, $\shd_X$ denoting the  filtered ring of differential operators (see~\cite{Ka03}). 
As it is well-known, the  category of filtered modules over a filtered ring is not abelian,  only exact in the sense of Quillen~\cite{Qu73} or quasi-abelian in the sense of~\cite{Sn99}, but this is enough to consider the derived category 
(see~\cite{BBD82,La83}). 
However, quasi-abelian categories are not easy to manipulate, and we shall show in this paper how 
to substitute a very natural abelian category to this quasi-abelian category, giving the same derived category. 

More precisely, consider an abelian category $\shc$ admitting small exact filtrant (equivalently, ``directed'')
 colimits and a filtrant preordered set  $\mon$.
In this paper, we regard a filtered object in $\shc$ as a functor $M\cl\mon\to\shc$ with the property that 
all $M(\lambda)$ are sub-objects  of $\sindlim M$. 
We prove that the derived category of the quasi-abelian category of filtered objects in $\shc$ indexed by $\mon$ is equivalent to the derived category of the abelian category of functors from $\mon$ to $\shc$. 
 Note that a particular case of this result, in which $\mon=\Z$ and $\shc$ is the category of abelian groups, was already obtained in~\cite[\S~3.1]{Sn99}. 

Next, we assume that $\shc$ is a tensor category and $\mon$ is a preordered semigroup. 
In this case, we can define what is a filtered ring $A$ indexed by $\mon$ and a filtered $A$-module in $\shc$ and we prove a similar result to the preceding one, namely that the derived category of the category of filtered $A$-modules is equivalent to the derived category of the abelian category of modules over the $\mon$-ring $A$. 

Applications to the study of filtered $\shd_X$-modules will be developed in the future. 
Indeed it is proved in~\cite{GS12} that,  on a complex manifold $X$ endowed with the subanalytic topology $\Xsa$, the sheaf 
$\sho_\Xsa$ (which is in fact an object of the derived category of sheaves, no more concentrated in degree zero) may be endowed 
with various filtrations and the results of this paper will be used when developing this point. 

\section{A review on quasi-abelian categories}\label{sec:abel}
In this section, we briefly review the main notions on quasi-abelian categories 
and their derived categories, after~\cite{Sn99}. 
We refer to \cite{KS06} for an exposition on abelian, triangulated and 
derived categories.

Let $\shc$ be an additive category admitting kernels and cokernels.
Recall that, for a morphism $f\cl X\to Y$ in $\shc$,
$\im(f)$ is the kernel of $Y\to\coker(f)$,
and $\coim(f)$ is the cokernel of $\ker(f)\to X$.
Then $f$ decomposes as 
\eqn
&&X\to\coim(f)\to\im(f)\to Y.
\eneqn
One  says that $f$ is {\em strict}
if $\coim(f)\to\im(f)$ is an isomorphism.
Note that a monomorphism (resp.{} an epimorphism) $f\cl X\to Y$ is strict 
if and only if $X\to \im(f)$
(resp.{} $\coim(f)\to Y$) is an isomorphism.
For any morphism $f\cl X\to Y$,
\begin{itemize}
\item
$\ker(f)\to X$ and  $\im(f)\to Y$ are strict monomorphisms,
\item
$X\to\coim(f)$ and $Y\to\coker(f)$ are strict epimorphisms.
\end{itemize}
Note also that a morphism $f$ is strict if and only if it factors as
$i\circ s$ with a strict epimorphism $s$ and a strict monomorphism $i$.

\begin{definition}\label{def:qabcat}
A {\em quasi-abelian} category is an additive category which admits
kernels and cokernels and satisfies the following conditions:
\bnum
\item
strict epimorphisms are stable by base changes,
\item
strict monomorphisms are stable by co-base changes.
\enum
\end{definition}
The condition (i) means that, for any strict epimorphism
$u\cl X\to Y$ and any morphism $Y'\to Y$,
setting $X'=X\times_{Y}Y'=\ker(X\times Y'\to Y)$,
the composition $X'\to X\times Y'\to Y'$ is a strict epimorphism.
The condition (ii) is  similar by reversing the arrows.

Note that, for any morphism $f\cl X\to Y$ in a quasi-abelian category,
$\coim(f)\to \im(f)$ is both a monomorphism and
an epimorphism.

Remark that if $\shc$ is a quasi-abelian category,
then its opposite category $\shc^\rop$ is also quasi-abelian.

Of course, an abelian category is quasi-abelian category category in which all morphisms are strict.

\begin{definition}
Let $\shc$ be a  quasi-abelian category.  
A sequence $M' \to[f] M \to[f'] M''$ with $f'\comp f=0$ 
is strictly exact if  $f$ is strict and the canonical morphism
$\im f \to \ker f'$ is an isomorphism.
\end{definition}
Equivalently such a  sequence is strictly exact if the canonical morphism
$\coim f \to \ker f'$ is an isomorphism. 

One shall be aware that the notion of strict exactness is not auto-dual.

Consider  a functor $F\cl\shc\to\shc'$ of quasi-abelian categories. Recall that $F$  is 
\begin{itemize}
\item
strictly exact if it sends any strict exact sequence 
$X'\to X\to X''$ to a strict exact sequence,
\item
strictly left exact if it sends any strict exact sequence  $0\to X'\to X\to X''$ to a strict exact sequence
$0\to F(X')\to F(X)\to F(X'')$,
\item 
left exact if it sends any strict exact sequence  $0\to X'\to X\to X''\to 0$ to a strict exact sequence
$0\to F(X')\to F(X)\to F(X'')$.
\end{itemize}

\subsubsection*{Derived categories}

Let $\shc$ be an additive category. One denotes as usual 
by $\RC(\shc)$ the additive category consisting of complexes in $\shc$.
For $X\in\RC(\shc)$, one denotes by 
$X^n$ ($n\in\Z$) its $n$'s component and by 
$d_X^n\colon X^n\to X^{n+1}$ the differential. 
For $k\in\Z$, one denotes by $X\mapsto X[k]$  the  shift functor 
(also called translation functor).
We denote by $\RC^+(\shc)$ (resp.\ $\RC^-(\shc)$, $\RC^\rb(\shc)$) the
full subcategory of 
$\RC(\shc)$ consisting of objects $X$ such that $X^n=0$ for 
$n\ll0$ (resp.\ $n\gg0$, $\vert n\vert\gg0$). One also sets 
$\RC^\ub(\shc)\eqdot \RC(\shc)$ ($\ub$ stands for unbounded).

We do not recall here neither the construction of the mapping cone
$\Mc(f)$ of a morphism $f$ in $\RC(\shc)$ nor the construction of 
the triangulated categories
$\rK^*(\shc)$ ($*=\ub,+,-,\rb$), called the homotopy categories of $\shc$. 

Recall that a null system $\shn$ in a triangulated category $\sht$ is
a full triangulated saturated subcategory of $\sht$, saturated
meaning that an object $X$ belongs to
$\shn$ whenever $X$ is isomorphic to an object of $\shn$. 
For a null system $\shn$, the localization $\sht/\shn$ is a
triangulated category. 
A distinguished triangle $X\to Y\to Z\to X[1]$ in $\sht/\shn$ 
is a triangle 
isomorphic to the image of
a distinguished triangle in $\sht$.

Let $\shc$ be quasi-abelian category.
One says that a complex $X$ is  
\begin{itemize}
\item
{\em strict} if all the differentials $d_X^n$ are strict,
\item
{\em strictly exact in degree $n$ } 
if the sequence $X^{n-1}\to X^n\to X^{n+1}$ is strictly exact.
\item
{\em strictly exact} 
if it is strictly exact in all degrees.
\end{itemize}
If $X$ is strictly exact, then
$X$ is a strict complex and $0\to \ker(d_X^{n})\to X^n\to \ker(d_X^{n+1})\to 0$
is strictly exact for all $n$.

Note that if two complexes $X$ and $Y$ are isomorphic in $\rK(\shc)$,
and if $X$ is strictly exact,
then so is $Y$. Let $\she$ be the full additive subcategory of $\rK(\shc)$
consisting of strictly exact complexes.
Then $\she$ is a null system in $\rK(\shc)$.

\begin{definition}
The derived category $\RD(\shc)$ is the quotient category $\rK(\shc)/\she$.
where $\she$ is the null system in $\rK(\shc)$ consisting of strictly exact complexes.
One defines similarly the categories $\RD^*(\shc)$ for $*=+,-,\rb$.
\end{definition}
A morphism $f\cl X\to Y$ in $\rK(\shc)$ is called a {\em quasi-isomorphism}
(a $\qis$ for short) if, after being embedded in a distinguished triangle
$X\to[\,f\,] Y\to Z\to X[1]$, $Z$ belongs to $\she$.  This is
equivalent to saying that its image in $\RD(\shc)$ is an isomorhism.
It follows that given morphisms $X\to[f]Y\to[g]Z$ in $\rK(\shc)$,
if two of $f$, $g$ and $g\circ f$ are $\qis$, then all the three are $\qis$.

Note that if $X\to[f]Y\to[g]Z$ is a sequence of morphisms in $\RC(\shc)$
such that $0\to X^n\to Y^n\to Z^n\to 0$ is strictly exact for all $n$, then
the natural morphism $\Mc(f)\to Z$ is a $\qis$, and we have a distinguished triangle 
\eqn
&& X\to Y\to Z\to X[1]
\eneqn
in $\RD(\shc)$.

\subsubsection*{Left $t$-structure}\label{subsec:t}
Let $\shc$ be a quasi-abelian category.
Recall that for $n\in\Z$,  $\RD^{\leq n}(\shc)$ (resp.\ $\RD^{\geq n}(\shc)$) denotes the full subcategory of $\RD(\shc)$ consisting of complexes $X$ which are strictly exact in degrees $k>n$ (resp.\ $k<n$). 
Note that $\RD^+(\shc)$ (resp.{} $\RD^-(\shc)$) is the union of
all the $\RD^{\ge n}(\shc)$'s (resp.{} all the $\RD^{\le n}(\shc)$'s), 
and $\Derb(\shc)$ is the intersection
$\RD^+(\shc)\cap\RD^-(\shc)$.
The associated  truncation functors are then given by:
\eqn
\hs{0ex}
\begin{array}{lccccccccccccccc}
\tau^{\le n}X\cl \cdots   &\to&X^{n-2}&\to& X^{n-1}            &\to&\ker d_X^n&\to&0            &\to&\cdots\\[1ex]
\tau^{\ge n}X\cl \cdots  &\to&0         &\to&\coim d_X^{n-1}&\to&X^n          &\to& X^{n+1}&\to&\cdots.
\end{array}
\eneqn
The functor $\tau^{\le n} \cl\RD(\shc)\to\RD^{\le n}(\shc)$
is a right adjoint to the inclusion functor
$\RD^{\le n}(\shc)\hookrightarrow \RD(\shc)$,
and $\tau^{\ge n} \cl\RD(\shc)\to\RD^{\ge n}(\shc)$
is a left adjoint functor to the inclusion functor $\RD^{\ge n}(\shc)\hookrightarrow\RD(\shc)$.

The pair $(\RD^{\le 0}(\shc),\RD^{\ge0}(\shc))$ defines a t-structure on $\RD(\shc)$ by~\cite{Sn99}. 
We refer to~\cite{BBD82} for the general theory of $t$-structures (see also \cite{KS90} for an exposition).

The heart $\RD^{\le 0}(\shc)\cap\RD^{\geq0}(\shc)$ is an abelian category 
called the left heart of $\RD(\shc)$ and denoted $\LH(\shc)$  in~\cite{Sn99}.
The embedding $\shc\hookrightarrow\LH(\shc)$ induces an equivalence
\eqn
&&\RD(\shc)\isoto\RD(\LH(\shc)).
\eneqn
By duality, one also defines the right $t$-structure and the right heart of $\RD(\shc)$.

\subsubsection*{Derived functors}
Given an additive functor $F\cl\shc\to\shc'$ of quasi-abelian categories, its right or left derived functor is defined in~\cite[Def.~1.3.1]{Sn99} by the same procedure as for abelian categories.

\begin{definition}{\rm (See~\cite[Def.~1.3.2]{Sn99}.)}\label{def:Fprojectcat}
A full additive subcategory $\shp$ of $\shc$ is called $F$-projective if
\banum
\item
for any $X\in\shc$, there exists a strict epimorphism $Y\to X$ with $Y\in\shp$,
\item
for any strict exact sequence $0\to X'\to X\to X''\to 0$ in $\shc$, if $X,X''\in\shp$, then 
$X'\in\shp$, 
\item
for any strict exact sequence $0\to X'\to X\to X''\to 0$ in $\shc$, if $X',X,X''\in\shp$, then
the sequence 
$0\to F(X')\to F(X)\to F(X'')\to 0$ in strictly exact in $\shc'$.
\eanum
\end{definition}
If $F$ admits an $F$-projective category, one says that $F$ is explicitly left derivable.
In this case, $F$ admits a left derived functor $LF$ and this functor is calculated as usual by 
the formula
\eqn
&& LF(X)\simeq F(Y), \text{ where } Y\in\rK^-(\shp), Y\isoto X \text{ in }\RD^-(\shc).
\eneqn
We refer to~\cite[\S 1.3]{Sn99} for details. 

If $LF$ has bounded cohomological dimension, then it extends as a triangulated functor 
$LF\cl\RD(\shc)\to\RD(\shc')$.

\section{Filtered objects}\label{section:filtobj}
We shall assume 
\eq\label{hyp:hypsection3}
&&\left\{\parbox{60ex}{
$\mon$ is a small filtrant category,\\
$\shc$ is an abelian
category admitting  small inductive limits and filtrant such limits are exact. 
}\right.\eneq

Denote by 
$\Fct(\mon,\shc)$ the abelian category of functors from $\mon$ 
to $\shc$, and denote as usual by $\Delta\cl \shc\to \Fct(\mon,\shc)$ the functor which, to $X\in\shc$,
associates the constant functor $\lambda\mapsto X$ and by $\sindlim\cl \Fct(\mon,\shc)\to\shc$ the inductive limit functor. 
Then $(\sindlim, \Delta)$ is a pair of  adjoint functors:

If $M\in\Fct(\mon,\shc)$, we set for short  $M(\infty)\eqdot\sindlim M$ and we denote by 
$j_M(\lambda)$ the morphism $M(\lambda)\to M(\infty)$. If $f\cl M\to M'$ is a morphism in 
$\Fct(\mon,\shc)$, we denote fy $f(\infty)\cl M(\infty) \to M'(\infty)$ the associated morphism.

\begin{definition}
\banum
\item
The category $\Fil_\Lambda\shc$ of $\mon$-filtered 
objects in $\shc$ is the full additive subcategory of $\Fct(\mon,\shc)$ formed by the functors which send morphisms to monomorphisms. 
\item
We denote by $\iota\cl\Fil_\Lambda\shc\into\Fct(\mon,\shc)$
the inclusion functor.
\eanum
\end{definition}
Inductive limits in $\shc$ being exact, the morphisms $j_M(\lambda)$ are monomorphisms. 

\begin{remark}\label{rem:poset}
(i) Let $M$ be a $\mon$-filtered object of $\shc$ and let $\lambda,\lambda'\in\mon$. Since
$j_M(\lambda)\comp M(s) = j_M(\lambda')$
for any morphism $s:\lambda'\to\lambda$ of $\Lambda$ and since $j_M(\lambda)$ is a monomorphism, it is clear that $M(s)$ does not depend on $s$. It follows that the category $\Fil_\mon(\shc)$ is equivalent to the category $\Fil_{\mon_{\text{pos}}}(\shc)$ where $\mon_{\text{pos}}$ denotes the category corresponding to the preordered set associated with $\mon$, \emph{i.\ e.} the category having the same objects as $\mon$ but for which
\eqn
	\Hom[\mon{\text{pos}}](\lambda',\lambda) = \left\{\begin{array}{ccc}
	\rmptt&\text{ if }&\Hom[\mon](\lambda',\lambda) \neq \emptyset,\\
	\emptyset&\text{ if }&\Hom[\mon](\lambda',\lambda) = \emptyset.
	\end{array}\right.
\eneqn
Therefore, when we study the properties of $\Fil_\mon(\shc)$ we can always assume that $\mon$ is the category associated with a preordered set.

\vspace{0.3ex}\noindent
(ii) When $\mon$ is a preordered set, $M$ defines an increasing map from $\mon$ to the poset of subobjects of $M(\infty)$. 
Moreover, $M(\infty)$ is the union of the $M(\lambda)$'s and we recover the usual notion of an exhaustive filtration. 
\end{remark}

\subsubsection*{Basic properties of $\Fil_\mon(\shc)$}
In this subsection, we shall prove that the  category $\Fil_\mon(\shc)$   is quasi-abelian. 
The next result is obvious. 

\begin{proposition}\label{prp:iota}
The subcategory $\Fil_\mon(\shc)$ of $\Fct(\mon,\shc)$ is stable by subobjects. 
In particular, the category $\Fil_\mon(\shc)$ admits kernels and 
the functor $\iota$ commutes with kernels.
\end{proposition}

\begin{definition}\label{def:kappa}
Let $M\in\Fct(\Lambda,\shc)$. For $\lambda\in\mon$, we set
$\kappa(M)(\lambda) = \im j_M(\lambda)$
and for a morphism $s\cl\lambda'\to\lambda$ in $\mon$ we define
$\kappa(M)(s)$ as the morphism induced by the identity of $M(\infty)$. 
\end{definition}
These definitions turn obviously $\kappa(M)$ into an object of $\Fil_\mon(\shc)$ and give a functor
\eq\label{eq:fctkappa}
&&\kappa : \Fct(\mon,\shc) \to \Fil_\mon(\shc).
\eneq

\begin{proposition}\label{prp:kappa}
The functor $\kappa$ in~\eqref{eq:fctkappa}
is  left adjoint to the inclusion functor $\iota$
and $\kappa\comp\iota\simeq\id_{\Fil_\mon(\shc)}$. 
In particular the category  $\Fil_\mon(\shc)$ admits cokernels and $\kappa$ commutes with cokernels.
\end{proposition}

\begin{proof}
Let $M$ be an object of $\Fil_\mon(\shc)$ and let $M'$ be an object of $\Fct(\Lambda,\shc)$. Consider a morphism
$f : M' \to \iota(M)$
in $\Fct(\Lambda,\shc)$. It induces a morphism $f(\infty): M'(\infty) \to M(\infty)$
in $\shc$. Since the diagram
\eqn
&&
\xymatrix{
M'(\lambda) \ar[d]_-{j_{M'}(\lambda)} \ar[r]_{f(\lambda)} & M(\lambda) \ar[d]^-{j_M(\lambda)}\\
M'(\infty) \ar[r]^{f(\infty)} & M(\infty) 
}
\eneqn
is commutative for every object $\lambda$ of $\Lambda$ and since $j_M(\lambda)$ is a monomorphism, the morphism $f(\lambda)$ induces a canonical morphism $f'(\lambda):\im j_{M'}(\lambda)\to M(\lambda)$
and these morphisms give us a morphism $f':\kappa(M')\to M$.
The preceding construction gives a morphism of abelian groups
\eqn
&&	\Hom[\Fct(\Lambda,\shc)](M',\iota(M)) \to \Hom[\Fil_\Lambda\shc](\kappa(M'),M)
\eneqn
and it is clearly an isomorphism. Since $\iota$ is fully faithful, we have 
$\kappa\comp\iota\simeq\id_{\Fil_\Lambda\shc}$ and the conclusion follows.
\end{proof}
By the preceding results, the category $\Fil_\mon(\shc)$ is additive and has kernels and cokernels and hence images and coimages. However, if $f\cl M'\to M$ is a morphism in $\Fil_\mon(\shc)$ the canonical morphism from $\coim(f)$ to $\im(f)$ is in general not an isomorphism, in other words, $f$ is not in general a strict morphism
and  the inclusion functor $\iota$ does not commute with  cokernels (see Example~\ref{exa:coimnotim} below).

\begin{corollary}\label{cor:strictexactness}
Let $f\cl M'\to M$ be a morphism in $\Fil_\mon(\shc)$  and let 
$\ker$, $\coker$, $\im$ and $\coim$ be calculated in the category $\Fil_\mon(\shc)$.
Then, one has the canonical isomorphisms 
for $\lambda\in\mon$:
\banum
\item 
$(\ker f)(\lambda)\simeq \ker f(\lambda)$,
\item
$(\coker f)(\lambda)\simeq \im [M(\lambda)\to\coker f(\infty)]$,
\item
$(\im f)(\lambda)\simeq \ker [M(\lambda)\to\coker f(\infty)]$,
\item
$(\coim f)(\lambda)\simeq \im f(\lambda)$.
\eanum
In particular, 
 \bnum
 \item
 $f$ is strict in $\Fil_\mon(\shc)$ if and only if the canonical square below is Cartesian
\eqn
\xymatrix{
\im f(\lambda) \ar[d] \ar[r] & \im f(\infty)\ar[d]\\
M(\lambda) \ar[r] & M(\infty),
}	
\eneqn
\item
a sequence $M' \to[f] M \to[f'] M''$ in $\Fil_\mon(\shc)$   with $f'\comp f=0$ is strictly exact 
if and only if the canonical morphism $\im f(\lambda) \to \ker f'(\lambda)$
is an isomorphism for any $\lambda\in\mon$.
\enum
\end{corollary}

\begin{example}\label{exa:coimnotim}
Let $\mon=\N$, $\shc=\md[\C]$ and denote by $\C[X]^{\leq n}$ the space of polynomials in one variable $X$ over $\C$ of degree $\leq n$. 
Consider the two objects $M'$ and $M$ of  $\Fil_\mon\shc$:
\eqn
&&M'\cl n\mapsto \C[X]^{\leq n},\\
&&M\cl n\mapsto \C[X]^{\leq n+1}.
\eneqn
Denote by $f\cl M'\to M$ the natural morphism.
Then $M'(\infty)\isoto M(\infty)\simeq \C[X]$ and $f(\infty)$ is an isomorphism. Therefore,
$(\im f)(n)\simeq M(n)$ and $(\coim f)(n)\simeq \im (f(n))\simeq M'(n)$. 
\end{example}

\begin{proposition}\label{prp:filteredobjectsarekappaacyclic}
Let $0\to M' \to[f] M \to[f'] M'' \to 0$
be an exact sequence in $\Fct(\mon,\shc)$. Assume that $M''$ belongs to $\Fil_\mon(\shc)$. Then the sequence
\[
	0\to \kappa(M') \to[\kappa(f)] \kappa(M) \to[\kappa(f')] \kappa(M'') \to 0
\]
is strictly exact in $\Fil_\mon(\shc)$ \lp\emph{i.\ e.} $\kappa(f)$ is a kernel of $\kappa(f')$ and $\kappa(f')$ is a cokernel of $\kappa(f)$\rp.
\end{proposition}

\begin{proof}
We know that the diagram
\[
\xymatrix{
0 \ar[r] & M'(\lambda) \ar[r]^{f(\lambda)}\ar[d]_-{j_{M'}(\lambda)} & M(\lambda) \ar[r]^{f'(\lambda)}\ar[d]_-{j_M(\lambda)} & M''(\lambda) \ar[r]\ar[d]_-{j_{M''}(\lambda)} & 0 \\
0 \ar[r] & M'(\infty) \ar[r]^{f(\infty)} & M(\infty) \ar[r]^{f'(\infty)} & M''(\infty) \ar[r] & 0 \\
}
\]
is commutative and has exact rows. Since the last vertical arrow is a monomorphism it follows that we have a canonical isomorphism
\[
\ker j_{M'}(\lambda) \simeq \ker j_{M}(\lambda).
\]
Therefore, in the commutative diagram
\[
\xymatrix{
         & 0\ar[d]                           & 0\ar[d]                         & 0\ar[d]                           &   \\
0 \ar[r] & \ker j_{M'}(\lambda) \ar[r]\ar[d] & \ker j_{M}(\lambda)\ar[r]\ar[d] & 0\ar[r]\ar[d]                     & 0 \\
0 \ar[r] & M'(\lambda) \ar[r]\ar[d]          & M(\lambda) \ar[r]\ar[d]         & M''(\lambda)\ar[r]\ar[d]          & 0 \\
0 \ar[r] & \kappa(M')(\lambda) \ar[r]\ar[d]  & \kappa(M)(\lambda) \ar[r]\ar[d] & \kappa(M'')(\lambda) \ar[r]\ar[d] & 0 \\
         & 0                                 & 0                               & 0                                 &
}
\]
the columns and the  two  lines in the top are exact. Therefore the last row is also exact and the conclusion follows from 
Corollary~\ref{cor:strictexactness}.
\end{proof}

\begin{theorem}\label{th:quasiab}
Assume~\ref{hyp:hypsection3}.
The category $\Fil_\mon(\shc)$ is quasi-abelian.
\end{theorem}

\begin{proof}
Consider a Cartesian square in $\Fil_\mon(\shc)$
\[
\xymatrix{
N'\ar[r]^{g}\ar[d]^{u'} & N \ar[d]^{u}\\
M'\ar[r]^{f}            & M            \\
}
\]
 and assume that $f$ is a strict epimorphism in $\Fil_\mon(\shc)$. It follows from Proposition~\ref{prp:iota} and Corollary~\ref{cor:strictexactness}~(ii) that this square is also Cartesian in $\Fct(\mon,\shc)$ and that $f$ is an epimorphism in this category. Hence $g$ is an epimorphism in $\Fct(\mon,\shc)$ and Corollary~\ref{cor:strictexactness}~(ii) shows that $g$ is a strict epimorphism in $\Fil_\mon(\shc)$.

Consider now a co-Cartesian square in $\Fil_\mon(\shc)$
\[
\xymatrix{
M'\ar[r]^{f}\ar[d]^{u'} & M \ar[d]^{u}\\
N'\ar[r]^{g}            & N           
}
\]
and assume that $f$ is a strict monomorphism in $\Fil_\mon(\shc)$. We know from Proposition~\ref{prp:iota} and Proposition~\ref{prp:kappa} that this square is the image by $\kappa$ of the co-Cartesian square of $\Fct(\mon,\shc)$ with solid arrow
\[
\xymatrix{
M'\ar[r]^{f}\ar[d]^{u'} & M \ar[d]^{v}\ar@{.>}[r]^-q&C\\
N'\ar[r]_-{h}            & P     \ar@{.>}[ru]_-{q'}&
}
\]
Denote by 
\[
	q:M \to C
\]
the canonical morphism from $M$ to the cokernel of $f$ in $\Fil_\mon(\shc)$. Since $f$ is a strict monomorphism of $\Fil_\mon(\shc)$, $C$ is also the cokernel of $f$ in $\Fct(\mon,\shc)$ and there is a unique morphism $q':P\to C$ such that 
$q'\comp v=q$ and $q'\comp h=0$. Moreover, one checks easily that this morphism $q'$ is a cokernel of $h$ in $\Fct(\mon,\shc)$. It follows that the sequence
\[
	0 \to N' \to[h] P \to[q'] C \to 0
\]
is exact in $\Fct(\mon,\shc)$. Applying $\kappa$ to this sequence and using Proposition~\ref{prp:filteredobjectsarekappaacyclic}, we get a strictly exact sequence of the form:
\[
	0 \to N' \to[g] N \to C \to 0.
\]
This shows in particular that $g$ is a strict monomorphism in $\Fil_\mon(\shc)$ and the conclusion follows.
\end{proof}

\subsubsection*{The Rees functor}

From now on we assume that $\mon$ is a category associated with a preordered set. Thanks to Remark~\ref{rem:poset}~(i),
this assumption is not really restrictive.

In the sequel, given a direct sum $\bigoplus_{i\in I}X_i$ we denote by $\sigma_i\cl X_i\to \bigoplus_{i\in I}X_i$ the canonical morphism.

\begin{definition}\label{def:rm}
For $M\in\Fct(\mon,\shc)$ we define $\Rees(M)\in\Fct(\mon,\shc)$ as follows.
For $\lambda_0\in\mon$ and for $s\cl\lambda_0\to\lambda_1$, we set
\eqn
&&\Rees(M)(\lambda_0) =\bigoplus_{s_0\cl\lambda'_0\to\lambda_0} M(\lambda'_0),
\eneqn
and define
\eqn
&&\Rees(M)(s)\cl\Rees(M)(\lambda_0)\to\Rees(M)(\lambda_1) 
\eneqn
as the only morphism such that $\Rees(M)(s)\comp\sigma_{s_0} =\sigma_{s\comp s_0}$
for any $s_0\cl\lambda'_0\to\lambda_0$.
\end{definition}

\begin{proposition}\label{prp:rfunctor}
Let $M\cl\mon\to\shc$ be a functor and let $s\cl\lambda_0\to\lambda_1$ be a morphism of $\mon$. Then
$\Rees(M)(s)\cl\Rees(M)(\lambda_0)\to\Rees(M)(\lambda_1)$
is a split monomorphism. In particular $\Rees(M)$ is an object of $\Fil_\mon\shc$.
\end{proposition}

\begin{proof}
Let us define $\rho\cl\Rees(M)(\lambda_1)\to\Rees(M)(\lambda_0)$
as the unique morphism such that
\[
	\rho\comp\sigma_{s_1} =\left\{\begin{array}{cl}
		\sigma_{s_0}&\text{ if }s_1=s\comp s_0\text{ for some } s_0\cl\lambda'_1\to\lambda_0\\
		0           &\text{ otherwise.}
	\end{array}
	\right.
\]
This definition makes sense since if $s_1=s\comp s_0$ for some $s_0\cl\lambda'_1\to\lambda_0$ then such an $s_0$ is 
unique (recall that $\mon$ is a poset).
Since 
\[
	\rho\comp\Rees(M)(s)\comp\sigma_{s_0}=\rho\comp\sigma_{s\comp s_0}=\sigma_{s_0}
\]
for any $s_0\cl\lambda'_0\to\lambda_0$ in $\mon$, the conclusion follows.
\end{proof}

\begin{remark}
The preceding construction gives rise to a functor, that we call the Rees functor, 
\eqn
&&\Rees\cl\Fct(\mon,\shc)\to\Fil_\mon\shc.
\eneqn
This functor sends exact sequences in $\Fct(\mon,\shc)$ to strict exact sequences in $\Fil_\mon\shc$.
\end{remark}

\begin{definition}
For any  $M\in\Fct(\mon,\shc)$ we define the morphism $\epsilon_M\cl\Rees(M)\to M$
by letting
\eq\label{eq:morepsilonM}
&&	\epsilon_M(\lambda_0)\cl\Rees(M)(\lambda_0)\to M(\lambda_0)
\eneq
be the unique morphism such that
$\epsilon_M(\lambda_0)\comp\sigma_{s_0} = M(s_0)$
for any $s_0\cl\lambda'_0\to\lambda_0$ in $\mon$.
\end{definition}

\begin{proposition}\label{prp:repsilon}
For any $M\in\Fct(\mon,\shc)$ and $\lambda_0\in\mon$, the morphism~\eqref{eq:morepsilonM}
is a split epimorphism of $\shc$. In particular, the morphism
$\epsilon_M\cl\Rees(M)\to M$
is an epimorphism in $\Fct(\mon,\shc)$.
\end{proposition}

\begin{proof}
This follows directly from
\[
	\epsilon_M(\lambda_0)\comp\sigma_{\id_{\lambda_0}} = M(\id_{\lambda_0}) =\id_{M(\lambda_0)},\, \lambda_0\in\mon.
\]
\end{proof}

\begin{corollary}
The category $\Fil_\mon\shc$ is a $\kappa$-projective subcategory of the category $\Fct(\mon,\shc)$. In particular the functor
\[
	\kappa\cl\Fct(\mon,\shc)\to\Fil_\mon\shc
\]
is explicitly left derivable. Moreover, it has finite cohomological dimension.
\end{corollary}

\begin{proof}
The properties (a), (b) and (c) of Definition~\ref{def:Fprojectcat} follow respectively from Proposition~\ref{prp:repsilon}, Proposition~\ref{prp:iota} and Proposition~\ref{prp:filteredobjectsarekappaacyclic}. Hence the category $\Fil_\mon\shc$ is 
$\kappa$-projective. Since it is also stable by subobjects it follows that any object of $\Fct(\mon,\shc)$ has a two terms resolution by objects of $\Fil_\mon\shc$ and the conclusion follows.
\end{proof}

\begin{theorem}\label{th:main1}
Assume~\eqref{hyp:hypsection3} and assume that $\mon$ is a preordered set.
The functor $\iota\cl\Fil_\mon\shc\to\Fct(\mon,\shc)$
is strictly exact and induces an equivalence of categories  for $*=\ub,\rb,+,-$
\[
 \iota\cl \RD^*(\Fil_\mon\shc)\to \RD^*(\Fct(\mon,\shc))
\]
whose quasi-inverse is given by
\[
 L\kappa\cl \RD^*(\Fct(\mon,\shc))\to \RD^*(\Fil_\mon\shc).
\]
Moreover, $\iota$ induces an equivalence of abelian categories
\[
	\LH(\Fil_\mon\shc)\simeq\Fct(\mon,\shc).
\]
\end{theorem}

\section{Filtered modules  in an abelian tensor category}

\subsubsection*{Abelian tensor categories}
In this subsection we recall a few facts about abelian tensor categories. References are made to~\cite[Ch.~5]{KS06} for details.

Let $\shc$ be an additive category. A biadditive tensor product on $\shc$ is  the data of  a functor
$\tens\cl\shc\times\shc\to\shc$
additive with respect to each argument together with functorial associativity isomorphisms
\[
	\alpha_{X,Y,Z}\cl(X\tens Y)\tens Z\isoto X\tens(Y\tens Z)
\]
satisfying the natural compatibility conditions. 

From now on, we assume that $\shc$ is endowed with such a tensor product. A ring object of $\shc$ 
(or, equivalently, ``a ring in $\shc$'') is then an object $A$ of $\shc$ endowed with an associative multiplication 
$\mu_A\cl A\tens A\to A$. 

Let $A$ be such a ring object. Then, an $A$-module of $\shc$ is the data of an object $M$ of $\shc$ together with an associative action $\nu_M\cl A\tens M\to M$.

The $A$-modules of $\shc$ form a category denoted $\Mod(A)$. A morphism $f\cl M\to N$ in this category is simply a morphism of $\shc$ which is $A$-linear, {\em i.\ e., } which is compatible with the actions of $A$ on $M$ and $N$. Most of the properties of $\Mod(A)$ can be deduced from that of $\shc$ thanks to following result, whose proof is  left to the reader.

\begin{lemma}\label{le:forgetmod}
The category $\Mod(A)$ is an additive category and the forgetful functor $\For\cl\Mod(A)\to\shc$
is additive, faithful, conservative and reflects projective limits. 
This functor also reflects inductive limits which are preserved by $A\tens\scbul\cl\shc\to\shc$.
\end{lemma}
One easily deduces:

\begin{proposition}\label{pro:forgetmod}
Assume that $\shc$ is abelian \lp resp.\ quasi-abelian\rp\, and  
 the functor $A\tens\scbul$ commutes with  cokernels. Then $\Mod(A)$ is abelian 
\lp resp.\ quasi-abelian\rp. Moreover, the forgetful functor
$\For\cl\Mod(A)\to\shc$
is additive, faithfull, conservative and commutes with  kernels and cokernels. 
If one assumes moreover that $\shc$ admits small inductive limits and that $A\tens\scbul$ commutes with such limits, then 
$\Mod(A)$ admits small inductive limits and the forgetful functor $\For$ commutes with such limits.
\end{proposition}
\begin{remark}
Let $A$ be a ring object in $\shc$.
We have defined an $A$-module by considering the left action of $A$. In other words, we have defined left $A$-modules. Clearly, one can defined right $A$-modules similarly, which is equivalent to replacing the tensor product $\tens$ with the opposite tensor product given by  $X\tens^\rop Y=Y\tens X$.
\end{remark}

 Assume that  the tensor category $\shc$ admits a unit, denoted  $\bfone$ and that the ring object $A$ also admits a unit
 $e\cl \bfone\to A$. In this case, we consider the full subcategory $\md[A_e]$ of $\md[A]$ consisting of modules 
 such that the action of $A$ is unital, which is translated by saying that the diagram below commutes:
 \eqn
 &&\xymatrix{
 \bfone\tens M\ar[r]^-{e\tens\id_M}\ar[rd]_-\sim&A\tens M\ar[d]^-{\nu_M}\\
 &M.
 }\eneqn 
 All results concerning $\md[A]$ still hold for $\md[A_e]$.

\begin{proposition}\label{pro:generator}
Let $\shc$ be  Grothendieck category with  a generator $G$ and assume that $\shc$ is also a tensor category 
with unit $\bfone$. 
Let  $A$ be a ring in $\shc$ with unit $e$ and assume that  the functor $A\tens\scbul$ commutes with  small inductive limits.      
Then $\md[A_e]$ is a  Grothendieck category and $A\tens G$ is a generator of this category.
\end{proposition}
\begin{proof}
For any $X\in\shc$ the morphism
\eqn
&&\bigoplus_{s\in\Hom[\shc](G,X)}G\to X
\eneqn
is an epimorphism. 
It follows that for any $M\in\md[A_e]$ the morphism
\eqn
&&\bigoplus_{s\in\Hom[\shc](G,M)}A\tens G\to A\tens M
\eneqn
is an epimorphism. Since $A$ has a unit, there is an epimorphism $A\tens M\epito M$. 
\end{proof}
\subsubsection*{$\mon$-rings and $\mon$-modules}

In this section, we shall assume
\eq\label{hyp:sect4a}
&&\left\{\parbox{60ex}{
$\mon$ is a filtrant preordered additive semigroup (viewed as a tensor category),\\
$\shc$ is an abelian tensor category which admits small inductive limits 
which   commute with $\tens$ and  small 
filtrant inductive limits are exact.
}\right.\eneq


\begin{definition}
For  $M_1,M_2\in\Fct(\mon,\shc)$ we define $M_1\tens M_2\in\Fct(\mon,\shc)$ as follows.
For  $\lambda,\lambda'\in\mon$ and  $s\cl\lambda'\to\lambda$ we set
\eqn
&&(M_1\tens M_2)(\lambda) =\indlim[\lambda_1+\lambda_2\leq\lambda] M_1(\lambda_1)\tens M_2(\lambda_2),
\eneqn
and we define $(M_1\tens M_2)(s) \cl (M_1\tens M_2)(\lambda')\to (M_1\tens M_2)(\lambda)$
to be the morphism induced by the inclusion
\eqn
&&\{(\lambda'_1,\lambda'_2)\in\mon\times\mon\cl\lambda'_1+\lambda'_2\leq\lambda'\}
\subset
\{(\lambda_1,\lambda_2)\in\mon\times\mon\cl\lambda_1+\lambda_2\leq\lambda\}.
\eneqn
\end{definition}

\begin{proposition}\label{pro:tensmontens}
The functor
\[
	\tens\cl\Fct(\mon,\shc)\times\Fct(\mon,\shc)\to\Fct(\mon,\shc)
\]
defined above turns  $\Fct(\mon,\shc)$ into a tensor category. Moreover, it commutes with small inductive limits.
\end{proposition}

\begin{proof}
The fact that the functor commutes with small inductive limits
follows from its definition. The  associativity follows from the associativity of the tensor product in $\shc$ 
and the canonical isomorphisms
\eqn
((M_1\tens M_2)\tens M_3)(\lambda) 
&\simeq&\indlim[\lambda_1+\lambda_2+\lambda_3\leq\lambda](M_1(\lambda_1)\tens M_2(\lambda_2))\tens M_3(\lambda_3),\\
(M_1\tens(M_2\tens M_3))(\lambda) 
&\simeq&\indlim[\lambda_1+\lambda_2+\lambda_3\leq\lambda]M_1(\lambda_1)\tens(M_2(\lambda_2))\tens M_3(\lambda_3)).
\eneqn
\end{proof}

\begin{definition}
\banum
\item
A \emph{$\mon$-ring of $\shc$} is  a ring of the tensor category $\Fct(\mon,\shc)$ considered in 
Proposition~\ref{pro:tensmontens}.
\item
A \emph{$\mon$-module of $\shc$} over a $\mon$-ring $A$ of $\shc$ is an $A$-module of the tensor category 
$\Fct(\mon,\shc)$. 
\item
As usual, we denote by $\Mod(A)$ the category of $A$-modules, that is, $\mon$-modules in $\shc$ over the $\mon$-ring $A$.
\eanum
\end{definition}

\begin{remark}\label{rem:multitens}
It follows from the preceding definitions that a $\mon$-ring of $\shc$ is the data of a functor 
$A\cl\mon\to\shc$ together with a multiplication morphism
\[
	A(\lambda_1)\tens A(\lambda_2)\to A(\lambda_1+\lambda_2)
\]
functorial in $\lambda_1,\lambda_2\in\mon$ and 
associative in a natural way. Moreover, if $A$ is such a ring then a $\mon$-module of $\shc$ over $A$ is the data of a functor 
$M\cl\mon\to\shc$ together with a functorial associative action morphism
\[
	A(\lambda_1)\tens M(\lambda_2)\to M(\lambda_1+\lambda_2).
\]
\end{remark}

\begin{remark}\label{rem:unitinmonrings}
Assume that the semigroup  $\mon$ admits a unit, denoted $0$ (in which case, one says that $\mon$ is a monoid), and the tensor category $\shc$ admits a unit, denoted 
 $\bfone$. Then the tensor category $\Fct(\mon,\shc)$ admits a unit, still denoted $\bfone_\mon$, defined as follows:
 \eqn
 &&\bfone_\mon(\lambda)=\left\{\begin{array}{ll}
 \bfone&\text{if }\lambda\geq0,\\
 0&\text{ otherwise.}
 \end{array}
 \right.
 \eneqn
 In such a case, the notion of a $\mon$-ring $A$ with unit $e$ makes sense as well as the notion of  an $A_e$-module. 
 \end{remark}

\begin{proposition}
Let $A$ be a $\mon$-ring of $\shc$. Then, the category $\Mod(A)$ is abelian and admits small inductive limits.
Moreover, the forgetful functor
$\For\cl\Mod(A)\to\Fct(\mon,\shc)$
is additive, faithfull, conservative and commutes with kernels and small inductive limits. In particular it is exact.
\end{proposition}

\begin{proof}
This follows directly from the preceding results and Proposition~\ref{pro:forgetmod}
\end{proof}

\subsubsection*{Filtered rings and modules}

\begin{definition}
We define the functor
\[
	\tens_F\cl\Fil_\mon\shc\times\Fil_\mon\shc\to\Fil_\mon\shc
\]
by the formula
\[
	M_1\tens_F M_2 =\kappa(\iota(M_1)\tens\iota(M_2))
\]
where the tensor product in the right-hand side is the tensor product of $\Fct(\mon,\shc)$.
\end{definition}

\begin{proposition}\label{prp:kappatens}
There is a canonical isomorphism
\[
	\kappa(M_1\tens M_2)\simeq\kappa(M_1)\tens[F]\kappa(M_2)
\]
functorial  in $M_1,M_2\in\Fct(\mon,\shc)$.
\end{proposition}

\begin{proof}
We know that for $\lambda, \lambda_1,\lambda_2\in\mon$, $\kappa(M_1\tens M_2)(\lambda)$
is the image of the canonical morphism $(M_1\tens M_2)(\lambda)\to(M_1\tens M_2)(\infty)$
and that $\kappa(M_1)(\lambda_1)$ and $\kappa(M_2)(\lambda_2)$ are respectively the images of the 
canonical morphisms $M_1(\lambda_1)\to M_1(\infty)$ and $M_2(\lambda_2)\to M_2(\infty)$. 
Since the morphisms
\[
	M_1(\lambda_1)\to\kappa(M_1)(\lambda_1)
	\quad\text{and}\quad
	M_2(\lambda_2)\to\kappa(M_2)(\lambda_2)
\]
are epimorphisms, so is the morphism
\[
	M_1(\lambda_1)\tens	M_2(\lambda_2)\to\kappa(M_1)(\lambda_1)\tens\kappa(M_2)(\lambda_2).
\]
It follows that the canonical morphism
\[
	\indlim[\lambda_1+\lambda_2\leq\lambda]M_1(\lambda_1)\tens M_2(\lambda_2) 
	\to 
	\indlim[\lambda_1+\lambda_2\leq\lambda]\kappa(M_1)(\lambda_1)\tens\kappa(M_2)(\lambda_2).
\]
is also an epimorphism. Since
\eqn
(M_1\tens M_2)(\infty) &\simeq& M_1(\infty)\tens M_2(\infty) 
	\simeq\kappa(M_1)(\infty)\tens\kappa(M_2)(\infty)\\
	&\simeq&\kappa(M_1\tens M_2)(\infty) 
\eneqn
the conclusion follows.
\end{proof}

\begin{proposition}
The functor $\tens_F\cl\Fil_\mon\shc\times\Fil_\mon\shc\to\Fil_\mon\shc$  turns  $\Fil_\mon\shc$ into a tensor category.
Moreover $\tens_F$ commutes with small inductive limits.
\end{proposition}

\begin{proof}
By the preceding result we have
\eqn
\kappa((\iota(M_1)\tens\iota(M_2))\tens\iota(M_3))&\simeq&\kappa(\iota(M_1)\tens\iota(M_2))\tens_F\kappa(\iota(M_3))\\
	&\simeq&(M_1\tens_F M_2)\tens_F M_3,\\
\kappa(\iota(M_1)\tens(\iota(M_2)\tens\iota(M_3)))
	&\simeq&\kappa(\iota(M_1))\tens_F\kappa(\iota(M_2)\tens\iota(M_3))\\
	&\simeq& M_1\tens_F(M_2\tens_F M_3).
\eneqn
Hence the associativity of $\tens_F$ follows from that of the tensor product of $\Fct(\mon,\shc)$. Since $\kappa$ 
commutes with small inductive limits, a similar argument shows that $\tens_F$ has the same property.
\end{proof}

\begin{definition}
\banum
\item
A \emph{$\mon$-filtered ring of $\shc$} is a ring object in the tensor category $\Fil_\mon\shc$. 
\item
A \emph{$\mon$-filtered module} $FM$ over a $\mon$-filtered ring $FA$, or simply, an $FA$-module $FM$, 
is an $FA$-module in the tensor category $\Fil_\mon\shc$. 
\item
As usual, we denote by $\Mod(FA)$ the category of $FA$-modules. 
\eanum
\end{definition}

\begin{remark}
It follows from the preceding definitions that $\Mod(FA)$ is the full subcategory of $\Mod(A)$ formed by the functors which send morphisms of $\mon$ to monomorphisms of $\shc$. The multiplication on $FA$ and the action of $FA$ on a module $FM$ may be described as in Remark~\ref{rem:multitens}.
\end{remark}

\begin{proposition}
Let $FA$ be a $\mon$-filtered ring of $\shc$. 
The category $\Mod(FA)$ is quasi-abelian and admits small inductive limits.
Moreover, the forgetful functor $\For\cl\Mod(FA)\to\Fil_\mon\shc$
is additive, faithfull, conservative and commutes with kernels and inductive limits.
\end{proposition}

\begin{proof}
This follows directly from the preceding results and Proposition~\ref{pro:forgetmod}
\end{proof}

\begin{proposition}
Let $FA$ be a $\mon$-filtered ring of $\shc$. Then $A\eqdot\iota FA$ is a $\mon$-ring of $\shc$ and the functors
$\iota\cl\Fil_\mon(\shc)\to\Fct(\mon,\shc)$ and $\kappa\cl\Fct(\mon,\shc)\to\Fil_\mon(\shc)$
induce functors 
\[
	\iota_A\cl\Mod(FA)\to\Mod(A)
	\quad\text{and}\quad
	\kappa_A\cl\Mod(A)\to\Mod(FA).
\]
Moreover $\kappa_A$ is  a left adjoint of $\iota_A$.
\end{proposition}

\begin{proof}
This follows easily from  Proposition~\ref{prp:kappatens} and the fact that $\kappa$ is a left adjoint of $\iota$.
\end{proof}

\begin{proposition}
Let $FA$ be a $\mon$-filtered ring of $\shc$ and set $A\eqdot\iota FA$. Let $M$ be an $A$-module. Then
 the functor $\Rees(M)$ of Definition~\ref{def:rm} has a canonical structure of an $A$-module and the morphism 
 $\epsilon_M\cl\Rees(M)\to M$ is $A$-linear.
\end{proposition}

\begin{proof}
We define the action of $FA$ on $\Rees(M)$ as the composition of the morphisms
\begin{align}
A(\lambda_1)\tens\Rees(M)(\lambda_2)
&= A(\lambda_1)\tens\bigoplus_{s_2\cl\lambda'_2\to\lambda_2} M(\lambda'_2)\notag\\
&\simeq\bigoplus_{s_2\cl\lambda'_2\to\lambda_2}A(\lambda_1)\tens M(\lambda'_2)\tag{*}\\
&\to\bigoplus_{s_2\cl\lambda'_2\to\lambda_2}M(\lambda_1+\lambda'_2)\tag{**}\\
&\to[v]\bigoplus_{s_3\cl\lambda'_3\to\lambda_1+\lambda_2}M(\lambda'_3)\tag{***}\\
&=\Rees(M)(\lambda_1+\lambda_2)\notag
\end{align}
where (*) comes from the fact that $\tens$ commutes with small inductive limits, (**) comes from the action of $A$ on $M$ and (***) is characterized by the fact that
\[
	v\comp\sigma_{s_2} =\sigma_{\id_{\lambda_1}+s_2}
\]
where $\id_{\lambda_1}+s_2\cl\lambda_1+\lambda'_2\to\lambda_1+\lambda_2$ is the map induced by 
$s_2\cl\lambda'_2\to\lambda_2$.
It is then easily verified that this action turns $\Rees(M)$ into an $A$-module for which the morphism 
$\epsilon_M\cl\Rees(M)\to M$ becomes $A$-linear.
\end{proof}

The following results can now be obtained by working as in Section~\ref{section:filtobj}.

\begin{corollary}
Let $FA$ be a $\mon$-filtered ring of $\shc$. Then the category $\Mod(FA)$ is a $\kappa_A$-projective subcategory of 
$\md[A]$. In particular the functor
\[
	\kappa_A\cl\Mod(A)\to\Mod(FA)
\]
is explicitly left derivable. Moreover, it has finite cohomological dimension.
\end{corollary}

\begin{theorem}\label{th:main2}
Assume~\eqref{hyp:sect4a}.
The functor
\[
	\iota_A\cl\Mod(FA)\to\Mod(A)
\]
is strictly exact and induces an equivalence of categories for $*=\ub,\rb,+,-$:
\[
 \iota_A\cl \RD^*(\md[FA])\to \RD^*(\md[A]) 
\]
whose quasi-inverse is given by
\[
 L\kappa_A\cl \RD^*(\md[A])\to \RD^*(\md[FA]).
\]
Moreover, $\iota_A$ 
induces an equivalence of abelian categories
\[
	\LH(\md[FA])\simeq\md[A].
\]
\end{theorem}

\begin{remark}\label{rem:Groth}
Assume that the semigroup $\mon$ admits a unit, denoted $0$, and the tensor category $\shc$ admits a unit, denoted 
 $\bfone$. Then the unit $\bfone_\mon$ of the category $\Fct(\mon,\shc)$ 
 (see Remark~\ref{rem:unitinmonrings}) belongs to $\Fil_\mon\shc$ and is a unit in this tensor category.
 In such a case, the notion of a $\mon$-filtered ring $FA$ with unit $e$ makes sense as well as the notion of $FA_e$-module. 
 
 Moreover, the results of Theorem~\ref{th:main2} remain true with $\md[FA]$ and $\md[A]$ replaced with 
 $\md[FA_e]$ and $\md[A_e]$, respectively. 
 
Assume moreover that $\shc$ is a Grothendieck category. In this case, $\md[A_e]$ is again a Grothendieck 
category by Proposition~\ref{pro:generator}.
\end{remark}

\subsubsection*{Example: modules over a filtered sheaf of rings}
Let $X$ be a site and let $\cor$ be a commutative unital algebra with finite global dimension. Consider the category 
$\shc=\md[\cor_X]$ of sheaves of $\cor_X$-modules and its derived category, $\RD(\cor_X)$. 
Let $\mon$ be as in~\eqref{hyp:sect4a}.

\begin{definition}
\banum
\item
A $\mon$-filtered sheaf $F\shf$ is a sheaf $\shf\in\md[\cor_X]$ endowed with a family of subsheaves 
$\{F_\lambda\shf\}_{\lambda\in\mon}$ 
such that $F_{\lambda'}\shf\subset F_{\lambda}\shf$ for any pair $\lambda'\leq\lambda$ and $\bigcup_jF_\lambda \shf=\shf$. 
\lp Of course,  the union $\bigcup$ is taken in the category of sheaves.\rp
\item
A $\mon$-filtered sheaf of $\cor_X$-algebras $F\shr$  is a filtered sheaf such that the underlying sheaf $\shr$ is a sheaf of 
$\cor_X$-algebras and $F_\lambda\shr\tens F_{\lambda'}\shr\subset F_{\lambda+\lambda'}\shr$ for all $\lambda,\lambda'\in\mon$. 
 (In particular, $F_0\shr$ is a subring of $\shr$.)
 \item
Given $F\shr$ as above, a left filtered module $F\shm$ over $F\shr$ is filtered sheaf such that the underlying sheaf $\shm$
is a sheaf of modules over $\shr$ and  $F_\lambda\shr\tens F_{\lambda'}\shm\subset F_{\lambda+\lambda'}\shm$ for all 
$\lambda,\lambda'\in\mon$.
\item
If  $\shr$ is unital, we ask that the unit of $\shr$ is a section of 
$F_0\shr$ and acts as the identity on each $F_{\lambda}\shm$.
\eanum
\end{definition}
The category $\md[F\shr]$ of filtered modules over $F\shr$ is quasi-abelian. 

On the other-hand, an object 
$F\shn$ of the abelian category $\md[\iota F\shr]$ is the data of a family of sheaves $\{F_\lambda\shn\}_{\lambda\in\mon}$, morphisms 
$F_\lambda\shn\to F_{\lambda'}\shn$ for any pair $\lambda\leq\lambda'$ and morphisms 
$F_\lambda\shr\tens F_{\lambda'}\shn\to F_{\lambda+\lambda'}\shn$ for all  $\lambda,\lambda'\in\mon$ satisfying the natural compatibility conditions but we do not ask any more that 
$F_{\lambda'}\shn$ is a subsheaf of $F_{\lambda}\shn$ for  $\lambda'\leq\lambda$.

By Theorem~\ref{th:main2}, we have an equivalence of categories for $*=\ub,\rb,+,-$:
\eqn
&&\RD^*(\md[F\shr])\isoto\RD^*(\md[\iota F\shr]).
\eneqn

\begin{example}
Let $(X,\sho_X)$ be a complex manifold and let $\shd_X$ be the sheaf of finite order differential operators. 

We apply the preceding construction to the tensor category $\md[\C_X]$. 
For $j\in\Z$, we denote by $F_j\shd_X$ the subsheaf of $\shd_X$ whose sections are differential operators of order $\leq j$, 
with $F_j\shd_X=0$ for $j<0$, 
and we  denote by $F\shd_X$ the ring $\shd_X$ endowed with this filtration. Recall that a filtered left $\shd_X$-module $F\shm$ 
is a left $\shd_X$-module $\shm$ endowed with a family of subsheaves $F_j\shm$ ($j\in\Z$) and morphisms 
$F_i\shd_X\tens F_j\shm\to F_{i+j}\shm$ satisfying natural compatibility conditions 
(the $F_j\shm$'s are thus $\sho_X$-modules) and such that $\bigcup_jF_j\shm=\shm$. Therefore, 
$F\shd_X$ is a $\Z$-ring in $\md[\C_X]$ and $F\shm$ is an $F\shd_X$-module, that is, an object of 
$\md[F\shd_X]$. 
\end{example}
\providecommand{\bysame}{\leavevmode\hbox to3em{\hrulefill}\thinspace}

\vspace*{1cm}
\noindent
\parbox[t]{20em}
{\scriptsize{
\noindent
Pierre Schapira\\
Institut de Math{\'e}matiques,
Universit{\'e} Pierre et Marie Curie\\
and\\
Mathematics Research Unit, \\
University of Luxemburg\\
e-mail: schapira@math.jussieu.fr\\
http://www.math.jussieu.fr/\textasciitilde schapira/

\vspace{2ex}\noindent
Jean-Pierre Schneiders\\
Institut de Math\'ematiques\\
Universit\'e de Li\`ege\\
email: jpschneiders@ulg.ac.be\\
http://www.analg.ulg.ac.be/jps/
}}

\end{document}